\documentclass{amsart}
\newtheorem{theorem}{Theorem}[section]
\newtheorem{lemma}[theorem]{Lemma}

\newtheorem{proposition}[theorem]{Proposition}
\theoremstyle{definition}
\newtheorem{definition}[theorem]{Definition}

\theoremstyle{remark}

\numberwithin{equation}{section}

\begin{document}

\title{Explicit  coproduct formula\\ for quantum groups of infinite series}
\author{V.K. Kharchenko}
\address{Universidad Nacional Aut\'onoma de M\'exico, Facultad de Estudios Superiores 
Cuautitl\'an, Primero de Mayo s/n, Campo 1, CIT, Cuautitlan Izcalli, Edstado de M\'exico, 54768, MEXICO}
\email{vlad@unam.mx}
\thanks{The author was supported by PAPIIT IN 112913, UNAM, and PACIVE CONS-18, FES-C UNAM, M\'exico.}

\subjclass{Primary 16W30, 16W35; Secondary 17B37.}

\keywords{Quantum group, Coproduct, PBW basis.}

\begin{abstract}
We find an explicit form for the coproduct formula for PBW generators 
of quantum groups of infinite series $U_q(\frak{sp}_{2n})$ and  $U_q(\frak{so}_{2n}).$
Similar formulas for $U_q(\frak{sl}_{n+1})$ and  $U_q(\frak{so}_{2n+1})$ are already known.
\end{abstract}
\maketitle
\markboth{V.K. Kharchenko}{Coproduct formula}

\centerline{\bf To appear in the Israel Journal of Mathematics}

\section{introduction}
In the present paper, we prove an explicit  coproduct formula 
for quantum groups $U_q(\mathfrak{g}),$ where 
$\mathfrak{g}=\mathfrak{sp}_{2n}$ or   $\mathfrak{g}=\mathfrak{so}_{2n}$
are simple Lie algebras of type $C,$ $D$ respectively.
Consider a Weyl basis of the Lie algebra $\mathfrak{g},$ 
$$
u[k,m]=[\ldots [[x_k,x_{k+1}],x_{k+2}], \ldots , x_m],
$$ 
see  \cite[Chapter VI, \S 4]{Ser} or  \cite[Chapter IV, \S 3, XVII]{Jac}.
Here, $x_i=x_{2n-i}$ and in case $C_n,$ we have $k\leq m\leq 2n-k,$ whereas 
 in case $D_n,$ the sequence $x_1,x_2, \ldots , x_{2n-1}$ has no term $x_{n-1}$ and $k\leq m<2n-k.$ 
If we replace the Lie operation by  skew brackets,
then the above basis becomes a set of PBW generators for the
related quantum group $U_q(\mathfrak{g}).$ We then find the coproduct of those PBW generators: 
\begin{equation}
\Delta (u[k,m])=u[k,m]\otimes 1+g_{km}\otimes u[k,m]
\label{c}
\end{equation}
$$
+\sum _{i=k}^{m-1}\tau _i(1-q^{-1})g_{ki}\, u[i+1,m]\otimes u[k,i],
$$
where  $g_{ki}$ is a group-like element that corresponds to $u[k,i],$ and  almost all $\tau $ equal $1.$ 
More precisely, in case $C_n,$ there  is one exception: 
$\tau _{n-1}=1+q^{-1}$ if $m=n.$ In case $D_n,$ the exception is:
$\tau _{n-1}=0$ if $m=n;$ and $\tau _{n-1}=p_{n\, n-1}$ otherwise.

 Recall that the same formula is valid for $U_q(\mathfrak{sl}_{n+1})$
and $U_q(\mathfrak{so}_{2n+1}).$ In case $A_n$ there are no 
exceptions \cite[Lemma 3.5]{KA}. In case $B_n,$ the main parameter $q$ becomes 
$q^2,$ and we have an exception $\tau_n=q,$ whereas in the sequence $x_1, x_2, \ldots , x_{2n},$
the variable $x_n$ appears twice: $x_i=x_{2n-i+1},$ see  \cite[Theorem 4.3]{Kh11}.
In case $B_2,$ an explicit formula  was established by  M. Beattie, S. D\v{a}sc\v{a}lescu, \c{S}. Raianu, \cite{BDR}.

 In the formula, if $i\geq 2n-m,$ then  $u[i+1,m]$ does not appear in the list
of the above PBW generators because $m>2n-(i+1).$ The elements 
$u[k,m]$ with $m>$ $2n-k$ are defined in a similar manner,
$$u[k,m]=[x_k,[x_{k+1}, \ldots , [x_{m-1},x_m]\ldots ]].$$ The formula  
remains valid for those elements as well, in which case
all of the $\tau $ equal $1,$ except for $\tau _n=1+q^{-1}$ if $k=n$
in case $C_n,$ and  $\tau _n=0$ if $k=n$ in case $D_n$ (by definition, in case $D_n,$
the sequence that  defines $u[n,m]$ has the form $x_n, x_{n+2}, x_{n+3}, \ldots , x_m).$
In other words, whereas the PBW generators
do not span a subcoalgebra, the formula remains valid for 
a basis of the subcoalgebra generated by them. Furthermore, the formula 
demonstrates that the PBW generators span a left coideal.

We are reminded  that M. Rosso \cite{Ros0} and H. Yamane \cite{Yam}
separately  constructed PBW generators for $U_q(\mathfrak{sl}_{n+1}).$
 Then, G. Lusztig \cite{Lus} found PBW bases for arbitrary $U_q(\mathfrak{g})$ in terms 
of his famous automorphisms defining the action of braid groups. A coproduct formula  for PBW generators $E_{\beta }$
 in Lusztig form appeared in the paper by S.Z. Levendorski  and Ya. S. Soibelman \cite[Theorem 2.4.2]{LS90}:
\begin{equation}
\Delta (E_{\beta }^n)-(E_{\beta }\otimes 1+q^{H\beta }\otimes E_{\beta })^n
\in U_h(\mathfrak{n}_+)_{\beta }\otimes U_h({\mathfrak B}_+).
\end{equation}
Recently I. Heckenberger and H.-J. Schneider \cite[Theorem 6.14]{HS} proved
a similar formula within a more general context: 
\begin{equation}
\Delta _{\mathfrak{B}(N)}(x)-x\otimes 1
\in {\bf k}\langle N_{{\beta }_{l-1}}\rangle {\bf k}\langle N_{{\beta }_{l-2}}\rangle 
\cdots {\bf k}\langle N_{{\beta }_1}\rangle\otimes \mathfrak{B}(N), \ \ x\in N_{\beta _l}.
\end{equation}
Although these formulas have no explicit form,
they are convenient for inductive considerations, particularly in the study of one-sided 
coideal subalgebras. 

We develop the coproduct formula by the same method as that in \cite{Kh11}
for the case $B_n.$ Firstly, we demonstrate that the values of the elements $u[k,m]$ in $U_q(\mathfrak{g})$
are almost independent of the arrangement of brackets (Lemmas \ref{ins}, \ref{ins1},  \ref{Dins}, \ref{Dins1}).
Then, using this fact, we demonstrate that these values form a set of PBW generators (Propositions \ref{strB}, \ref{DstrB}).
Next, we find the explicit shuffle representation of those elements  (Propositions \ref{shu}, \ref{Dshu}). 
In case $C_n$ (as well as in cases $A_n$ and $B_n)$ these PBW generators are proportional to shuffle comonomials. 
This proportionality makes it easy to find the coproduct of those elements inside the  shuffle coalgebra.
Because there is a clear connection  (\ref{copro}) between the
coproduct in $U_q(\mathfrak{g})$ and the coproduct in the shuffle coalgebra, 
we can set up the coproduct formula (Theorem \ref{cos}). In case $D_n,$ each PBW generator 
is either proportional to a comonomial or a  linear  combination of two comonomials. 
These two options  allows one to find the coproduct inside the shuffle coalgebra
and deduce the coproduct formula (Theorem \ref{Dcos}).

The set of PBW generators for  $U_q(\mathfrak{g})$ is the union of those sets for
positive and negative quantum Borel subalgebras. Thus, we 
focus only on the positive quantum Borel subalgebra $U_q^+(\mathfrak{g}).$
\section{Preliminaries}

\subsection{Skew brackets}
Let $X=$ $\{ x_1, x_2,\ldots, x_n\} $ be a set of quantum variables; that is, 
associated with each $x_i$ there are an element $g_i$ of a fixed Abelian group 
$G$ and a character $\chi ^i:G\rightarrow {\bf k}^*.$ 
For every word $w$ in $X,$ let $g_w$ or gr$(w)$ denote
an element of $G$ that appears from $w$ by replacing each $x_i$ with $g_i.$
In the same manner,  $\chi ^w$ denotes a character that appears from $w$
by replacing each $x_i$ with $\chi ^i.$

Let $G\langle X\rangle $ denote the skew group algebra generated by $G$
and {\bf k}$\langle X\rangle $ with the commutation rules $x_ig=\chi ^i(g)gx_i,$
or equivalently $wg=\chi ^w(g)gw,$ where $w$ is an arbitrary word in $X.$
If $u,$ $v$ are homogeneous in  each $x_i,$ $1\leq i\leq n$ polynomials,
then the skew brackets are defined by the formula 
\begin{equation}
[u,v]=uv-\chi ^u(g_v) vu.
\label{sqo}
\end{equation}
We use the notation  $\chi ^u(g_v)=p_{uv}=p(u,v).$
The form $p(\hbox{-},\hbox{-})$ is bimultiplicative: 
\begin{equation}
p(u, vt)=p(u,v)p(u,t), \ \ p(ut,v)=p(u,v)p(t,v).
\label{sqot}
\end{equation}
In particular $p(\hbox{-},\hbox{-})$ is completely defined by $n^2$ parameters $p_{ij}=\chi ^{i}(g_{j}).$ 

The brackets satisfy an analog of the Jacobi identity:
\begin{equation}
[[u, v],w]=[u,[v,w]]+p_{wv}^{-1}[[u,w],v]+(p_{vw}-p_{wv}^{-1})[u,w]\cdot v.
\label{jak1}
\end{equation}
The antisymmetry identity transforms as follows:
\begin{equation}
[u,v]=-p_{uv}[v,u]+(1-p_{uv}p_{vu})u\cdot v
\label{cha}
\end{equation}
The Jacobi identity (\ref{jak1}) implies a conditional  identity:
\begin{equation}
[[u, v],w]=[u,[v,w]],\hbox{ provided that } [u,w]=0.
\label{jak3}
\end{equation}
By the evident induction on length,  this result allows for the following generalization:
\begin{lemma} \cite[Lemma 2.2]{KL}.
If $y_1,$ $y_2,$ $\ldots ,$ $y_m$ are homogeneous linear combinations of words such that
$[y_i,y_j]=0,$ $1\leq i<j-1<m,$ then  the bracketed polynomial $[y_1y_2\ldots y_m]$ is independent 
of the precise arrangement of brackets:
\begin{equation}
[y_1y_2\ldots y_m]=[[y_1y_2\ldots y_s],[y_{s+1}y_{s+2}\ldots y_m]], \ 1\leq s<m.
\label{ind}
\end{equation}
\label{indle}
\end{lemma}
Another conditional identity is:
if $[u,v]=0$ (that is, $uv=p_{uv}vu$), then 
\begin{equation}
[u,[v,w]]=-p_{vw}[[u,w],v]+p_{uv}(1-p_{vw}p_{wv})v\cdot [u,w]. 
\label{jja}
\end{equation}
The brackets are related to the product by ad-identities:
\begin{equation}
[u\cdot v,w]=p_{vw}[u,w]\cdot v+u\cdot [v,w], 
\label{br1f}
\end{equation}
\begin{equation}
[u,v\cdot w]=[u,v]\cdot w+p_{uv}v\cdot [u,w].
\label{br1}
\end{equation}
It is easy to verify all of the identities developing the brackets by (\ref{sqo}).

\smallskip
\subsection{Quantum Borel algebra}
The group  $G$ acts on the free algebra ${\bf k}\langle X\rangle $
by $ g^{-1}ug=\chi ^u(g)u,$ where $u$ is an arbitrary monomial in $X.$
The skew group algebra $G\langle X\rangle $ has a natural Hopf algebra structure 
$$
\Delta (x_i)=x_i\otimes 1+g_i\otimes x_i, 
\ \ \ i\in I, \ \ \Delta (g)=g\otimes g, \ g\in G.
$$

Let $C=||a_{ij}||$ be a symmetrizable Cartan matrix and let 
$D={\rm diag }(d_1, \ldots , d_n)$ be such that $d_ia_{ij}=d_ja_{ji}.$
We denote a Kac-Moody algebra defined by $C,$ see \cite{Kac}, as $\mathfrak g.$
Suppose that parameters $p_{ij}$ are related by
\begin{equation}
p_{ii}=q^{d_i}, \ \ p_{ij}p_{ji}=q^{d_ia_{ij}},\ \ \ 1\leq i,j\leq n. 
\label{KM1}
\end{equation}
In this case the multiparameter quantization $U^+_q ({\mathfrak g})$ 
is a homomorphic image of $G\langle X\rangle $ defined by Serre relations 
with the skew brackets in place of the Lie operation:
\begin{equation}
[\ldots [[x_i,\underbrace{x_j],x_j], \ldots ,x_j]}_{1-a_{ji} \hbox{ times}}=0, \ \ 1\leq i\neq j\leq n.
\label{KM2}
\end{equation}
By \cite[Theorem 6.1]{Khar}, the left-hand sides of these relations are skew-primitive elements 
in $G\langle X\rangle .$ Therefore the ideal generated by  these elements is a Hopf ideal,
hence $U^+_q ({\mathfrak g})$  has the natural structure of a Hopf algebra.

\subsection{PBW basis.} Recall that a linearly ordered set $V$ is said to be
a {\it set of PBW generators} (of infinite heights) if the set of all products
\begin{equation}
g\cdot v_1^{n_1}\cdot v_2^{n_2}\cdot \ \cdots \ \cdot v_k^{n_k}, \ \ \ g\in G, \ \ v_1<v_2<\ldots <v_k\in V
\label{pbge}
\end{equation}
is a basis of $U_q^+(\mathfrak{g}).$ 

We fix the order $x_1>x_2>\ldots >x_n$  on the set $X.$
On the set of all words in $X,$ we fix the lexicographical order
with the priority from left to right,
where a proper beginning of a word is considered to 
be greater than the word itself.

A non-empty word $u$ is called a {\it standard Lyndon-Shirshov} word if $vw>wv$
for each decomposition $u=vw$ with non-empty $v,w.$
The {\it standard arrangement} of brackets $[u]$ on a standard
word $u$ is defined by induction: $[u]=[[v][w]],$
where $v, w$ are the standard words such that $u=vw$
and  $v$ has the minimal length \cite{pSh1}, \cite{pSh2}, see also \cite{Lot}.

In \cite{Kh4}, it was proven that the values of bracketed standard words 
corresponding to positive roots with the lexicographical order 
form a set of PBW generators (of infinite heights) for 
$U_q^+(\mathfrak{g}),$ where $\mathfrak{g}$ is a Lie algebra of infinite series $A, B, C, D.$

\smallskip
\subsection{Shuffle representation}

The {\bf k}-algebra $A$ generated by values of $x_i,$ $1\leq i\leq n$ in $U_q^+(\mathfrak{g})$
is not a Hopf subalgebra because it has no nontrivial group-like elements. 
Nevertheless, 
$A$ is a Hopf algebra in the category of Yetter-Drinfeld modules over {\bf k}$[G].$
In particular, $A$ has a structure of a braided Hopf algebra
with a braiding $\tau (u\otimes v)=p(v,u)^{-1}v\otimes u.$
The braided coproduct  $\Delta ^b:A\rightarrow A\underline{\otimes }A$ 
is connected with the coproduct on $U_q^+(\mathfrak{g})$ as follows
\begin{equation}
\Delta ^b(u)=\sum _{(u)}u^{(1)}\hbox{gr}(u^{(2)})^{-1}\underline{\otimes} u^{(2)}, \hbox{ where }\  
\Delta (u)=\sum _{(u)}u^{(1)}\otimes u^{(2)}.
\label{copro}
\end{equation}

The tensor space $T(W),$ $W=\sum x_i{\bf k}$ also has the structure of a braided Hopf algebra,
which is the {\it braided shuffle algebra} $Sh_{\tau }(W)$ with the coproduct 
\begin{equation}
\Delta ^b(u)=\sum _{i=0}^m(z_1\ldots z_i)\underline{\otimes} (z_{i+1}\ldots z_m),
\label{bcopro}
\end{equation}
where $z_i\in X,$ and $u=(z_1z_2\ldots z_{m-1}z_m)$ is the tensor
$z_1\otimes z_2\otimes \ldots \otimes z_{m-1}\otimes z_m,$ called a {\it comonomial},
considered as an element of $Sh_{\tau }(W).$ The braided shuffle product satisfies
\begin{equation}
(w)(x_i)=\sum _{uv=w}p(x_i,v)^{-1}(ux_iv), \ \ (x_i)(w)=\sum _{uv=w}p(u,x_i)^{-1}(ux_iv).
\label{spro}
\end{equation}
The map $x_i\rightarrow (x_i)$ defines a homomorphism of the braided Hopf algebra
$A$ into the braided Hopf algebra $Sh_{\tau }(W).$ This  is 
extremely useful for calculating the coproducts due to formulae (\ref{copro}) and (\ref{bcopro}).
If $q$ is not a root of $1,$ then this representation is faithful. Otherwise, 
its  kernel is  the largest Hopf ideal in $A^{(2)},$
where $A^{(2)}$ is the ideal of $A $
generated by values of $x_ix_j,$ $1\leq i,j\leq n.$ See details in P. Schauenberg  \cite{Sch}, M. Rosso \cite{Ros},
  M. Takeuchi \cite{Tak1},  D. Flores de Chela and J.A. Green  \cite{FC}, N. Andruskiewitsch, H.-J. Schneider  \cite{AS}, V. K. Kharchenko \cite{Kh03}.

\section{Relations in $U_q^+({\mathfrak sp}_{2n}).$}
Throughout the following three sections, we fix a parameter $q$ such that $q^3\not=1,$ $q\not= -1.$
If $C$ is a Cartan matrix of type $C_n,$ then relations (\ref{KM1}) take the form
\begin{equation}
p_{ii}=q, \ 1\leq i<n;\ \ p_{i\, i-1}p_{i-1\, i}=q^{-1}, \ 1<i<n;\  p_{ij}p_{ji}=1,\  j>i+1;
\label{b1rel}
\end{equation}
\begin{equation}
 \ p_{nn}=q^2, \ \ p_{n-1\, n}p_{n\, n-1}=q^{-2}. 
\label{b1rell}
\end{equation}
In this case,
the quantum Borel algebra $U^+_q (\mathfrak{sp}_{2n})$
is a homomorphic image of $G\langle X\rangle $ subject to the following relations 
\begin{equation}
[x_i,[x_i,x_{i+1}]]=[[x_i,x_{i+1}],x_{i+1}]=0, \ 1\leq i<n-1; \ \ [x_i,x_j]=0, \ \ j>i+1;
\label{relb}
\end{equation}
\begin{equation}
[[x_{n-1},x_n],x_n]=[x_{n-1},[x_{n-1},[x_{n-1},x_n]]]=0.
\label{relbl}
\end{equation} 
\begin{lemma} 
If $u$ is a standard word independent of $x_n,$ then either $u$ $=x_kx_{k+1}\ldots x_m,$ $k\leq m< n,$
or $[u]=0$ in $U_{q}^+(\mathfrak{sp}_{2n}).$ Here $[u]$ is a nonassociative word
with the standard arrangement of brackets.  
\label{nul}
\end{lemma}
\begin{proof}
The Hopf subalgebra of $U_q^+(\mathfrak{sp}_{2n})$ 
generated  by $x_i,$ $1\leq i<n$ is the Hopf algebra $U_{q}^+({\mathfrak sl}_{n})$ defined by the Cartan matrix of type $A_{n-1}.$
By this reason the third statement of \cite[Theorem $A_n$]{Kh4} applies.
\end{proof}

\begin{definition} \rm 
In what follows, $x_i,$ $n<i<2n$ denotes the generator $x_{2n-i}.$
Respectively, $v(k,m),$ $1\leq k\leq m<2n$ is the word
$x_kx_{k+1}\cdots x_{m-1}x_m,$ whereas $v(m,k)$ is the opposite word
$x_mx_{m-1}\cdots x_{k+1}x_k.$  If $1\leq i<2n,$ then $\phi (i)$
denotes the number $2n-i,$ so that $x_i=x_{\phi (i)}.$ 
\label{fis}
\end{definition}

\begin{definition} \rm 
If $k\leq i<m<2n,$ then we set
\begin{equation}
\sigma _k^m\stackrel{df}{=}p(v(k,m),v(k,m)),
\label{mu11}
\end{equation}
\begin{equation}
\mu _k^{m,i}\stackrel{df}{=}p(v(k,i),v(i+1,m))\cdot p(v(i+1,m),v(k,i)).
\label{mu1}
\end{equation}
\label{slo}
\end{definition}
\begin{lemma}
For each $i,$ $k\leq i<m$ we have 
\begin{equation}
\mu _k^{m,i}=\sigma _k^m(\sigma _k^i \sigma _{i+1}^m)^{-1}.
\label{mu23}
\end{equation}
\label{mu}
\end{lemma}
\begin{proof}
Because $p(\hbox{-},\hbox{-})$ is a bimultiplicative map, there is a decomposition
\begin{equation}
p(ab,ab)=p(a,a)p(b,b)\cdot p(a,b)p(b,a).
\label{mu55}
\end{equation}
Applying this equality to $a=v(k,i),$ $b=v(i+1,m),$ we get the required relation.
\end{proof}

\begin{lemma}
If $1\leq k\leq m<2n,$ then 
\begin{equation}
\sigma_k^m =\left\{ \begin{matrix}
q^2,\hfill &\hbox{if } m=\phi (k);\hfill \cr
q,\hfill &\hbox{otherwise}.\hfill 
                             \end{matrix}
                    \right.
\label{mu21}
\end{equation}
\label{sig}
\end{lemma}
\begin{proof}
The bimultiplicativity of $p(\hbox{-},\hbox{-})$ implies that  $\sigma _k^m=\prod _{k\leq s,t\leq m}p_{st}$
is the product of all coefficients of the $(m-k+1)\times (m-k+1)$-matrix $||p_{st}||.$ By (\ref{b1rel})
all coefficients on the main diagonal equal $q$ except $p_{nn}=q^2.$

If $m<n$ or $k>n,$ then for non-diagonal coefficients, we have 
$p_{st}p_{ts}=1$ unless $|s-t|=1,$ whereas  $p_{s\, s+1}p_{s+1\, s}=q^{-1}.$ 
Hence, $\sigma _k^m=q^{m-k+1}\cdot q^{-(m-k)}=q.$

If $m=n$ or $k=n$ but not both, then we have $p_{nn}=q^2,$ $p_{n\, n-1}p_{n-1\, n}=q^{-2}.$ By the above reasoning,
 we get $\sigma _k^m=q^{(m-k)+2}\cdot q^{-(m-k-1)-2}=q.$ Of course, if $k=n=m$ then  $\sigma _k^m=p_{nn}=q^2.$

In the remaining case,  $k<n<m,$ we use induction on $m-k.$

By (\ref{mu55}) we have
\begin{equation}
\sigma _k^{m+1}=\sigma _k^m\cdot q\cdot p(v(k,m),x_{m+1})\cdot p(x_{m+1}, v(k,m)).
\label{si69}
\end{equation}
We shall prove that if $k<n<m,$ then
\begin{equation}
p(v(k,m),x_{m+1})\cdot p(x_{m+1}, v(k,m))=\left\{ \begin{matrix}
1,\hfill &\hbox{if } k=\phi (m)-1;\hfill \cr
q^{-2},\hfill &\hbox{if } k=\phi (m);\hfill \cr
q^{-1},\hfill &\hbox{otherwise.}\hfill 
                             \end{matrix}
                    \right.
\label{si45}
\end{equation}
The left hand side of the above equality is $\prod_{k\leq t\leq m} p_{t\, m+1}p_{m+1\, t}.$
If $m=n,$ then by \ref{b1rel} and \ref{b1rell}, the factor $p_{t\, n+1}p_{n+1\, t}$ differs from 1 only
if $t\in \{ n-2, n-1, n \} $ and related values are respectively $q^{-1}, q^2, q^{-2}.$
 Hence, if $k<n-1=\phi (m)-1,$ then the total product is $q^{-1};$ if $k=n-1=\psi (m)-1,$ then this is $1;$
if $k=n=\phi (m),$ then this is $q^{-2}.$

If $m>n,$ then the factor $p_{t\, m+1}p_{m+1\, t}$ differs from 1 only
if $$t\in \{ \phi (m)-2, \phi (m)-1, \phi (m), m \}$$ and related values are respectively $q^{-1}, q^2, q^{-1}, q^{-1}.$
Therefore if $k<\phi (m)-1,$ then the whole product is $q^{-1};$
if $k=\phi (m)-1,$ then this is $1;$ if $k=\phi (m),$ then this is $q^{-2};$ if $k>\phi (m),$ then this is again $q^{-1}.$
This completes the proof of (\ref{si45}).

To complete the inductive step, we use (\ref{si45}) and inductive hypothesis:
if $k$ $=\phi (m)-1,$ then $\sigma _k^{m+1}$ $=q\cdot q\cdot 1$ $=q^2;$
 if $k$ $=\phi (m),$ then $\sigma _k^{m+1}$ $=q^2\cdot q\cdot q^{-2}$ $=q;$
otherwise $\sigma _k^{m+1}$ $=q\cdot q\cdot q^{-1}$ $=q.$
\end{proof}

\smallskip
We define the bracketing of $v(k,m),$ $k\leq m<2n$ as follows.
\begin{equation}
v[k,m]=\left\{ \begin{matrix} [[[\ldots [x_k,x_{k+1}], \ldots ],x_{m-1}], x_m],\hfill 
&\hbox{if } m<\phi (k);\hfill \cr
 [x_k,[x_{k+1},[\ldots ,[x_{m-1},x_m]\ldots ]]],\hfill &\hbox{if } m>\phi (k);\hfill \cr 
[\! [v[k,m-1],x_m]\! ],\hfill &\hbox{if } m=\phi (k),\hfill 
\end{matrix}\right.
\label{ww}
\end{equation}
where in the latter term,  $[\! [u,v]\! ]\stackrel{df}{=}uv-q^{-1}p(u,v)vu.$

Conditional identity (\ref{ind}) demonstrates that the value of  $v[k,m]$ in $U_q^+(\mathfrak{sp}_{2n})$
is independent of the precise arrangement of brackets, provided that $m\leq n$ or $k\geq n.$
Now we are going to analyze what happens  with the arrangement  of brackets if $k< n<m\neq \phi (k).$

\begin{lemma} 
If  $k\leq n\leq m<\phi (k),$ then the value in $U_q^+(\mathfrak{sp}_{2n})$ 
of the bracketed word $[y_kx_{n}x_{n+1}\cdots x_m],$ where $y_k=v[k,n-1],$
is independent of the precise arrangement of brackets. 
\label{ins}
\end{lemma}
\begin{proof} 
To apply (\ref{ind}), it suffices to check 
$[y_k,x_t]=0,$  where $n<t\leq m$ or, equivalently, $\phi (m)\leq t<n.$ 
By (\ref{jak3}) we have
$$
[y_k,x_t]=\hbox{\Large[}[v[k,t-2],v[t-1,n-1]],x_t\hbox{\Large]}
=\hbox{\Large[}v[k,t-2], [v[t-1,n-1],x_t]\hbox{\Large]}.
$$
By Lemma \ref{nul} the element $[v[t-1,n-1],x_t]$ equals zero in $U_{q}^+(\mathfrak{so}_{2n})$
because the word $u(t-1,n)x_t$ is independent of $x_n,$ it is standard, and the standard bracketing is precisely $[v[t-1,n],x_t].$ 
\end{proof} 
\begin{lemma} 
If $k\leq n,$ $ \phi (k)<m,$ then the value in $U_q^+(\mathfrak{sp}_{2n})$ 
of the bracketed word $[x_kx_{k+1}\cdots x_ny_m],$ where $y_m=v[n+1,m],$
is independent of the precise arrangement of brackets. 
\label{ins1}
\end{lemma}

\begin{proof}
To apply (\ref{ind}), we need the equalities $[x_t,y_m]=0,$ $k\leq t<n.$ 
The polynomial $[x_t,y_m]$ is independent of $x_n.$ 
Moreover,  $[x_t,y_m]$ is proportional to $[y_m, x_t]$ due to antisymmetry identity (\ref{cha}) because 
$$p(x_t, y_m)p(y_m,x_t)=p_{t\, t+1}p_{tt}p_{t\, t-1}\cdot p_{t+1\, t}p_{tt}p_{t-1\, t}=1.$$
The equality $[y_m, x_t]=0$ turns to the proved above equality $[v[k,n-1],x_t]=0$ if one renames  the variables
$x_{n+1}\leftarrow x_k,$ $x_{n+2}\leftarrow x_{k+1}, \ldots .$
\end{proof}

\section{PBW generators of $U_q^+(\mathfrak{sp}_{2n})$}
\smallskip

\begin{proposition} If $q^3\neq 1,$ $q\neq -1,$ then
values of the elements  $v[k,m],$ $k\leq m\leq \phi (k)$ form a set of PBW generators with infinite heights
for the algebra $U_q^+(\mathfrak{sp}_{2n})$ over {\bf k}$[G].$
\label{strB}
\end{proposition}
\begin{proof} 
A word $v(k,m)$  is a standard Lyndon-Shirshov word provided that $k$ $\leq m$ $<\phi (k)$. By 
 \cite[Theorem $C_n,$ p. 218]{Kh4} these words with the standard bracketing, say  $[v(k,m)],$  become 
a set of PBW generators if we add to them the elements $[v_k]\stackrel{df}{=}[v[k,n-1], v[k,n]]$ $1\leq k<n.$
We shall use induction on $m-k$  in order to demonstrate that the value in $U_q^+(\mathfrak{sp}_{2n})$
of $[v(k,m)],$ $k\leq m<\phi (k)$ 
is the same as the value of  $v[k,m]$ with the bracketing given in (\ref{ww}).

If $m\leq n,$ then the value of $v[k,m]$ is independent of the arrangement of brackets,
see Lemma \ref{indle}. 

If $k<n<m,$ then according to \cite[Lemma 7.18]{Kh4}, the brackets in $[v(k,m)]$ 
are set by the following recurrence formulae (we note that $[v(k, m)]=[v_{k\, \phi (m)}]$ in the notations of \cite{Kh4}):
\begin{equation}
\begin{matrix}
[v(k,m)]=[x_k[v(k+1, m)]], \hfill & \hbox{if  } m<\phi (k)-1; \hfill \cr
[v(k, m)]=[[v(k, m-1)]x_m], \hfill & \hbox{if  } m=\phi (k)-1.\hfill 
\end{matrix}
\label{wsk}
\end{equation}
In the latter case, the induction applies directly. In the former case, using induction and Lemma \ref{ins}, we have 
$$[v(k+1, m)]=v[k+1, m]=[v[k+1,n-1], v[n,m]].$$ At the same time $[x_k,x_t]=0,$  $n\leq t\leq m$ because $x_t=x_{\phi (t)}$
and $\phi (t)\geq \phi (m)>k+1.$
This implies $[x_k, v[n,m]]=0.$ Applying conditional identity (\ref{jak3}), we get 
$$
[v(k, m)]=[x_k[v[k+1,n-1], v[n,m]]]={\big [}[x_kv[k+1,n-1]], v[n,m]{\big ]}=v[k,m].
$$

It remains to analyze the case $m=\phi (k).$ We have to demonstrate that if
in the set $V$ of PBW generators of Lyndon-Shirshov standard words 
one replaces the elements $[v_k]$
with $v[k,\phi (k)],$ $1\leq k<n$ then the obtained set is still a set of PBW generators.
To do this, due to  \cite[Lemma 2.5]{KhT} with $T\leftarrow \{ v[k,\phi (k)], 1\leq k<n\} ,$
$S\leftarrow U_q(\mathfrak{sp}_{2n}),$ it suffices to see that the leading term
of the PBW decomposition of $v[k,\phi (k)]$ in the generators $V$ is proportional to $[v_k].$

By definition (\ref{ww}) with $m=\phi (k),$ we have 
$$
v[k,m]=v[k,m-1]x_m-q^{-1}\pi  x_mv[k,m-1]
$$
$$
=-q^{-1}\pi [x_m,v[k,m-1]]+(1-q^{-1}\pi \pi^{\prime})v[k,m-1]\cdot x_m,
$$
where $\pi =p(v(k,m-1),x_m),$ $\pi ^{\prime }=p(x_m,v(k,m-1)).$
The second term of the latter sum is a basis element (\ref{pbge}) in the PBW generators $V.$
This element starts with $v[k,m-1]$
which is lesser than $[v_k].$ Hence it remains to analyze the bracket $[x_k,v[k,m-1]].$

If $k=n-1,$ then $[x_k,v[k,m-1]]$ $=[x_{n-1}, [x_{n-1},x_{n}]]$ $=[v_k].$ 

If $k<n-1,$ then by Lemma \ref{ins} we have 
\begin{equation}
[x_k,v[k,m-1]]=[ x_k, [v[k,n], v[n+1, m-1] ] ].
\label{esk}
\end{equation}
The basic relations (\ref{relb}) imply
$[x_k,[x_k,x_{k+1}]]=0,$ $[x_k, v[k+2,n]]=0.$ By Lemma \ref{indle} value of $v[k,n]$ is
independent of the arrangement of brackets,  $$v[k,n]=[[x_k,x_{k+1}],v[k+2,n]],$$
hence $[ x_k, v[k,n] ]=0.$

By Eq. (\ref{jja}) with $u\leftarrow x_k,$ $v\leftarrow v[k,n]$, $w\leftarrow v[n+1, m-1],$
the right hand side of (\ref{esk}) is a linear combination of the following two elements:
\begin{equation}
[ x_k, v[n+1, m-1]],  v[k,n] ], \ \ \ \ v[k,n]\cdot [ x_k, v[n+1, m-1]].
\label{sk}
\end{equation}
The latter element starts  with  a factor $v[k,n]$ which is lesser than $[v_k].$ 
Hence it remains to prove that the leading term 
of the former element is proportional $[v_k].$

By downward induction on $k$ we shall prove the following decomposition
\begin{equation}
v[n+1, m-1]=\alpha \, v[k+1,n-1]+\sum_{s=k+2}^{n-1}\gamma _s \, v[s,n-1]\cdot U_s,\ \ \alpha \neq 0.
\label{desk}
\end{equation}
If $k=n-2,$ then this decomposition reduces to $x_{n+1}=x_{n-1}.$
Let us apply $[\hbox{-}, x_m]$ to the both sides of the above equality. 
Using (\ref{cha}), we see that  $[v[k+1,n-1],x_m]$ is proportional to $v[k,n-1]+\gamma _{k+1}\, v[k+1,n-1]\cdot x_{m},$
whereas (\ref{br1f}) implies $[v[s,n-1]\cdot U_s, x_m]$ $=v[s,n-1]\cdot [U_s, x_{k-1}]$ for $s\geq k+2.$
This completes the inductive step. 

Let us apply $[x_k, \hbox{-}]$ to both sides of the already proved Eq. (\ref{desk}).
By (\ref{br1}), we get 
\begin{equation}
[x_k,v[n+1, m-1]]=\alpha \, v[k,n-1]+\sum_{s=k+2}^{n-1}\gamma _s^{\prime } \, v[s,n-1]\cdot [x_k,U_s].
\label{de1}
\end{equation}

Finally, let us apply $[\hbox{-},v[k,n]]$ to both sides of (\ref{de1}). In this way we find a decomposition of the first element of $(\ref{sk})$:
\begin{equation}
[[x_k,v[n+1, m-1]], v[k,n]]=\alpha \, [v_k]+\sum_{s=k+2}^{n-1}\gamma _s^{\prime }
 \, v[s,n-1]\cdot [x_k,U_s]\cdot v[k,n]
\label{de11}
\end{equation}
$$
-\sum_{s=k+2}^{n-1}\gamma _s^{\prime \prime} \, v[k,n]\cdot v[s,n-1]\cdot [x_k,U_s].
$$
All summands, except the first one,
 start with  $v[k,n],$ $v[s,n-1]$ that are lesser than $[v_k].$
Hence, the leading term, indeed, is proportional to $[v_k].$
\end{proof}

\begin{proposition}
Let $k\leq m<2n.$ In the shuffle representation, we have
\begin{equation}
v[k,m]=\alpha _k^m\cdot (v(m,k)), \ \
\alpha _k^m\stackrel{df}{=}\varepsilon_k^m (q-1)^{m-k}\cdot \prod _{k\leq i<j\leq m}p_{ij},
\label{shur}
\end{equation}
where 
\begin{equation}
\varepsilon _k^m =\left\{ \begin{matrix}
1+q,\hfill &\hbox{if } k\leq n\leq m, m\neq \phi (k);\hfill \cr
1+q^{-1},\hfill &\hbox{if }  m=\phi (k)\neq n;\hfill \cr
1,\hfill &\hbox{otherwise.}\hfill 
                             \end{matrix}
                    \right.
\label{eps}
\end{equation}
\label{shu}
\end{proposition}
\begin{proof} 
We use induction on $m-k.$ If $m=k,$ then the equality reduces to $x_k=(x_k).$

a). Consider first the case $m<\phi (k).$ By the inductive supposition, we have
$v[k,m-1]=\alpha _k^{m-1}\cdot (w),$ $w=v(m-1,k).$ Using (\ref{spro}), we may write
$$
v[k,m]=\alpha _k^{m-1}\{ (w)(x_m)-p(w,x_m)\cdot (x_m)(w)\}
$$
\begin{equation}
=\alpha _k^{m-1}\sum _{uv=w}\{ p(x_m,v)^{-1}-p(v,x_m)\} (ux_mv),
\label{sum1}
\end{equation}
where $p(v,x_m)=p(w,x_m)p(u,x_m)^{-1}$ because $w=uv.$ 

If $m\leq n,$  then relations (\ref{b1rell}) imply $p(v,x_m)p(x_m,v)=1$ with only one exception 
being $v=w.$
Hence, sum (\ref{sum1}) has just one term. The coefficient of $(x_mw)=(v(m,k))$ equals 
$$
\alpha _k^{m-1}p(w,x_m)(p(w,x_m)^{-1}p(x_m,w)^{-1}-1)=\alpha _k^{m-1}\prod_{i=k}^{m-1}p_{im} \cdot 
(p_{m-1\, m}^{-1}p_{m\, m-1}^{-1}-1).
$$
If $m<n,$ then the latter factor equals $q-1,$ whereas if $m=n,$ then it is $q^2-1=\varepsilon_k^{n}(q-1).$

Suppose that $m>n$ and still $m<\phi (k).$ 
In decomposition (\ref{sum1}), we have $v=v(s,k),$ $k\leq s<m$ and hence
$p(x_m,v)p(v,x_m)=\prod\limits_{t=k}^{s}p_{mt}p_{tm}.$
The product  $p_{mt}p_{tm}$ differs from $1$ only if $t\in \{ \phi(m)-1, \phi (m), \phi (m)+1, m-1\};$
 related values are $q^{-1},$ $q^2,$ $q^{-1},$ $q^{-1}$ if $m>n+1,$ and they are 
$q^{-1},$ $q^2,$ $q^{-2}$ if $m=n+1,$ $\phi (m)+1=m-1.$ This implies 
\begin{equation}
p(x_m,v)p(v,x_m)=
\left\{ \begin{matrix}
q^{-1},\hfill &\hbox{ if } s=\phi (m)-1, \hbox{ or } s=m-1;\hfill \cr
q,\hfill &\hbox{ if }  s=\phi (m);\hfill \cr
1,\hfill &\hbox{otherwise.}\hfill 
                             \end{matrix}
                    \right.
\label{pups}
\end{equation}
Hence, in (\ref{sum1}), only three terms remain with
$s=\phi (m)-1,$ $s=\phi (m),$ and $s=m-1.$ If $s=\phi (m)-1$ or $s=\phi (m),$ then $(ux_mv)$
equals 
$$
ux_mv=v(m-1,\phi (m)+1)x_m^2v(\phi (m)-1,k),
$$
 whereas the coefficient of the comonomial $(ux_mv)$ in sum (\ref{sum1}) is
$$
p(x_m,v_0)^{-1}-p(v_0,x_m)+p(x_m,x_mv_0)^{-1}-p(x_mv_0, x_m),
$$
where $v_0=v(\psi (m)-1,k).$ Taking into account (\ref{pups}), we 
find the above sum:
$$
p(v_0,x_m)(q-1+q\cdot q^{-1}-q)=0.
$$
Thus, in (\ref{sum1}) only one term remains, with $v=v(m-1,k),$ $u=\emptyset $. This term has the required coefficient:
$$
\alpha _k^m=\alpha _k^{m-1}(p(x_m,w)^{-1}-p(w,x_m))=\alpha _k^{m-1}p(w,x_m)(q-1).
$$

b). Consider the case $m>\phi (k).$
By the inductive supposition, we have
$$v[k+1,m]=\alpha _{k+1}^{m}\cdot (w),\ \ \ w=v(m,k+1).$$ Using (\ref{spro}), we get
$$
v[k,m]=\alpha _{k+1}^m\{ (x_k)(w)-p(x_k,w)\cdot (w)(x_k)\}
$$
$$
=\alpha _{k+1}^{m}\sum _{uv=w}\{ p(u,x_k)^{-1}-p(x_k,w)p(x_k,v)^{-1}\} (ux_kv).
$$
\begin{equation}
=\alpha _{k+1}^{m}\sum _{uv=v(m,k+1)}p(x_k,u\{ p(u,x_k)^{-1}p(x_k,u)^{-1}-1\} (ux_kv).
\label{sum2}
\end{equation}

If $k\geq n,$  then $p(u,x_k)p(x_k,u)=1,$ unless $u=w.$
Hence, (\ref{sum2}) has only one term, and the coefficient  equals 
$$
\alpha _{k+1}^{m}p(x_k,w)(p(w,x_k)^{-1}p(x_k,w)^{-1}-1)=\alpha _{k+1}^{m}p(x_k,w)(p_{k+1\, k}^{-1}p_{k\, k+1}^{-1}-1).
$$
If $k>n,$ then the latter factor equals $q-1,$ whereas if $k=n,$ then it is $q^2-1$ $=(q-1)\varepsilon _n^m$ as claimed.

Suppose that $k<n.$ In this case, $x_k=x_t$ with $m>t\stackrel{df}{=} \phi (k)>\phi (n)=n.$ Let $u=v(m,s).$ 

If $s>t+1,$ then $u$ depends only on $x_i,$ $i<k-1,$ 
and relations (\ref{b1rel}), (\ref{b1rell}) imply $p(x_k,u)p(u,x_k)=1.$ 

If $s<t,$ $s\neq k+1,$ then $k+1<n$ (otherwise $s=n=k+1),$ and we have
$p(x_k,u)p(u,x_k)$ $=p_{k-1\, k}p_{kk}p_{k+1\, k}\cdot p_{k\, k-1}p_{kk}p_{k+1\, k}$ $=1$
because $x_t=x_k.$

Hence, three terms remain in (\ref{sum2})  with
$s=t,$ $s=t+1,$ and $s=k+1.$ If $u=v(m, t)$ or $u=v(m, t+1),$ then $ux_kv$ 
$=v(m,t+1)x_k^2v(t-1,k),$ whereas the coefficient of the corresponding  tensor is
$$
p(v(m,t+1),x_k)^{-1}-p(x_k,v(m,t+1))+p(v(m,t),x_k)^{-1}-p(x_k,v(m, t))
$$
$$
 =p(x_k,v(m, t+1))\{ p_{k-1\, k}^{-1}p_{k\, k-1}^{-1}-1+
p_{kk}^{-1}p_{k-1\, k}^{-1}p_{k\, k-1}^{-1}-p_{kk}\} =0
$$
because $p_{kk}=q,$ $p_{k-1\, k}p_{k\, k-1}=q^{-1},$ and $p_{kr}p_{rk}=1$ if $r>t+1.$

Thus, only one term remains in (\ref{sum1}), and
$$
\alpha _k^m=\alpha _{k+1}^m(p(w,x_k)^{-1}-p(x_k,w))=\alpha _{k+1}^mp(x_k,w)(q-1).
$$

c). Let $m=\phi (k)\neq n.$ In this case, $x_m=x_k,$ $p_{kk}=q.$ By definition (\ref{ww}) we have
\begin{equation}
v[k,m]=v[k,m-1]\cdot x_k-q^{-1} p(v(k,m-1),x_m)x_k\cdot v[k,m-1].
\label{kk}
\end{equation}
Case a)  allows us to find the shuffle representation
$$v[k,m-1]=\alpha _k^{m-1} (w),\ \ \ w=v(m-1,k).$$
Hence the right-hand side of (\ref{kk}) in the shuffle form is
$$
\alpha_k^{m-1} \sum _{uv=w}
\left( p(x_k,v)^{-1}- q^{-1} p(v(k,m-1),x_m) \cdot p(u, x_k)^{-1}\right) \cdot (ux_kv) 
$$
\begin{equation}
=\alpha_k^{m-1} \sum _{uv=v(m-1,k)} p(v,x_m)
\left( p(x_k,v)^{-1}p(v,x_k)^{-1}- q^{-1} \right) \cdot (ux_kv). 
\label{bli}
\end{equation}
The coefficient of  $(v(k,m))$ related to $u=\emptyset ,$ $v=v(m-1,k)$ equals 
\begin{equation}
\alpha_k^{m-1}p(v,x_m)\cdot (p(x_k,v)^{-1}p(v,x_k)^{-1}-q^{-1})  
=\alpha_k^{m-1}(1-q^{-1})\cdot \prod _{i=k}^{m-1}p_{im}.
\label{uje}
\end{equation}
Here we have used $x_k=x_m$ and
Eq. (\ref{si45}) with $m\leftarrow m-1,$ $k\leftarrow k.$
It remains to show that all other terms in (\ref{bli}) are canceled. In this case we would have
$$
\varepsilon _k^m=\varepsilon _k^{m-1} (1-q^{-1})(q-1)^{-1}=(1+q)(1-q^{-1})(q-1)^{-1}=1+q^{-1}
$$
 as required.

If  $u=v(m-1,k),$ $v=\emptyset $ or $u=v(m-1,k+1),$ $v=x_k,$ 
 then $$ux_kv=v(m-1,k+1)x_k^2,$$
whereas the total coefficient of the related comonomial is proportional to
$$
1-q^{-1}\cdot p(u,x_m)p(u,x_k)^{-1}+p_{kk}^{-1}-q^{-1}\cdot p_{kk}=0.
$$

Let $u=v(m-1,s),$ $v=v(s-1,k),$ $k+1<s<m.$
The whole coefficient of the comonomial $(ux_kv)$ takes the form
$$
\alpha _k^{m-1}p(v(s-1,k),x_m)\cdot \left( p(x_k, v(s-1,k))^{-1}p(v(s-1,k),x_k)^{-1}-q^{-1})\right) .
$$
The latter factor equals $\prod\limits_{t=k}^{s-1}p_{kt}^{-1}p_{tk}^{-1}-q^{-1}.$ The product 
$p_{kt}p_{tk}$ differs from $1$ only if $t\in \{ k, k+1\}$ and related values are
$q^2$ and  $q^{-1}.$ This implies that the coefficient of $(ux_kv)$ has a factor $q^{-2}\cdot q-q^{-1}=0.$
\end{proof} 

\section{Coproduct formula for $U_q^+({\mathfrak{sp}}_{2n})$}

\begin{theorem} In $U_q^+(\mathfrak{sp}_{2n})$
the coproduct on the elements $v[k,m],$ $k\leq m<2n$ 
 has the following explicit form
\begin{equation}
\Delta (v[k,m])=v[k,m]\otimes 1+g_{km}\otimes v[k,m]
\label{co}
\end{equation}
$$
+\sum _{i=k}^{m-1}\tau _i(1-q^{-1})g_{ki}\, v[i+1,m]\otimes v[k,i],
$$
where $\tau _i=1$ with two  exceptions, being $\tau _{n-1}=1+q^{-1}$ if $m=n,$
and $\tau _n=1+q^{-1}$ if $k=n.$ Here $g_{ki}=\hbox{\rm gr}(v(k,i))=g_kg_{k+1}\cdots g_i.$
\label{cos}
\end{theorem}
\begin{proof}
By Proposition \ref{shu} we have the shuffle representation
\begin{equation}
v[k,m]=\alpha _k^m\cdot (v(m,k)).
\label{Cco2}
\end{equation}
Using (\ref{bcopro}), it is easy to find the braided coproduct of the comonomial shuffle:  
$$
\Delta ^b_0((v(m,k))=\sum _{i=k}^{m-1}(v(m,i+1))\underline{\otimes }(v(i,k)),
$$
where for short we put $\Delta _0^b(U)=\Delta^b(U)-U\underline{\otimes }1-1\underline{\otimes }U.$
Taking into account (\ref{Cco2}), we have
\begin{equation}
\Delta ^b_0(v[k,m])=\alpha _k^m\cdot \sum _{i=k}^{m-1} (\alpha _{i+1}^m)^{-1}v[i+1,m]\underline{\otimes }(\alpha _i^k)^{-1}v[k,i].
\label{co5}
\end{equation}

Formula  (\ref{copro}) demonstrates that the tensors $u^{(1)}\otimes u^{(1)}$ of the (unbraided) coproduct and tensors
$u^{(1)}_b\underline{\otimes } u^{(1)}_b$ of the braided one are related by 
$u^{(1)}_b=u^{(1)}{\rm gr}(u^{(2)})^{-1},\ $ $u^{(2)}_b=u^{(2)}.$ The equality  (\ref{co5})
provides the values of $u^{(1)}_b$ and $u^{(2)}_b.$ Hence we may find  
$u^{(1)}$ $=\alpha _k^m(\alpha _k^i\alpha _{i+1}^m)^{-1}$
$\cdot v[i+1,m]g_{ki}$ and $u^{(2)}=v[k,i],$ where $g_{ki}={\rm gr}(v[k,i]).$ 
The commutation rules imply 
$$
v[i+1,m]g_{ki}=p(v(i+1,m), v(k,i))g_{ki}v[i+1,m].
$$
Thus, the coproduct has the form (\ref{co}), where 
$$\tau _i(1-q^{-1})=\alpha _k^m(\alpha _k^i\alpha _{i+1}^m)^{-1}p(v(i+1,m), v(k,i)).$$
The definition of $\alpha _k^m$ given in (\ref{shur}) shows that 
$$
\alpha _k^m(\alpha _k^i\alpha _{i+1}^m)^{-1}=\varepsilon _k^m(\varepsilon_k^i\varepsilon_{i+1}^m)^{-1}\cdot p(v(k,i),v(i+1,m))
$$
because 
$$
\left( \prod _{k\leq a<b\leq i}p_{ab}\prod _{i+1\leq a<b\leq m}
p_{ab}\right) ^{-1}\prod _{k\leq a<b\leq m}p_{ab}=p(v(k,i),v(i+1,m)).
$$ 
The definition of $\mu _k^{m,i}$ given in  (\ref{mu1})
implies $$\tau _i(1-q^{-1})=\varepsilon _k^m(\varepsilon_k^i\varepsilon_{i+1}^m)^{-1}
(q-1)\mu _k^{m,i};$$ that is, 
$\tau _i=\varepsilon _k^m(\varepsilon_k^i\varepsilon_{i+1}^m)^{-1}q\mu _k^{m,i}.$ 
By (\ref{mu23}), we have $\mu _k^{m,i}=\sigma _k^m(\sigma _k^i\sigma _{i+1}^m)^{-1}.$
Using (\ref{mu21}) and (\ref{eps}), we see that
\begin{equation}
\varepsilon _k^m\sigma _k^m=\left\{ \begin{matrix}
q^2+q,\hfill &\hbox{if } k\leq n\leq m,  k\neq m;\hfill \cr
q^2,\hfill &\hbox{if } k=n=m;\hfill \cr
q,\hfill &\hbox{otherwise}.\hfill 
                             \end{matrix}
                    \right.
\label{tau1}
\end{equation}
Now, it is easy to check that   the
$\tau $'s have the following elegant form
\begin{equation}
\tau_i=\varepsilon _k^m\sigma _k^m(\varepsilon _k^i\sigma _k^i)^{-1}(\varepsilon _{i+1}^m\sigma _{i+1}^m)^{-1} q
\label{tau1}
\end{equation}
$$
=\left\{ \begin{matrix}
1+q^{-1},\hfill &\hbox{if } i=n-1, m=n; \hbox{ or } k=i=n;\hfill \cr
1,\hfill &\hbox{otherwise}.\hfill 
                             \end{matrix}
                    \right.
$$
\end{proof}

{\bf Remark 1.} \label{t} If $q$ is a root of $1,$ say $q^t=1,$ $t>2,$ then the shuffle representation
is not faithful. Therefore in this case, the formula (\ref{co}) is proved only 
for the Frobenius-Lusztig kernel $u_q(\mathfrak{sp}_{2n}).$ Nevertheless, all tensors in 
(\ref{co}) have degree at most 2 in each variable. At the same time, general results 
on combinatorial representation of Nichols algebras \cite[Section 5.5]{Ang} demonstrate that in case 
$C_n,$ the kernel of the natural projection 
$U_q(\mathfrak{sp}_{2n})\rightarrow u_q(\mathfrak{sp}_{2n})$
is generated by polynomials of degree grater then 2  in (or independent of) each given variable.
Hence (\ref{co}) remains valid in this case as well.

\section{Relations in $U_q^+(\mathfrak{so}_{2n})$}

In what follows, we fix a parameter $q$ such that $q\not= -1.$
If $C$ is a Cartan matrix of type $D_n,$ then relations (\ref{KM1}) take the form
\begin{equation}
p_{ii}=q, \ 1\leq i\leq n;\ \ p_{i\, i-1}p_{i-1\, i}=p_{n-2\, n}p_{n\, n-2}=q^{-1}, \ 1<i<n;\  
\label{Db1rel}
\end{equation}
\begin{equation}
p_{ij}p_{ji}=p_{n-1\, n}p_{n\, n-1}=1,\hbox{ if }  j>i+1\,  \& \, (i,j)\neq (n,n-2).
\label{Db1rell}
\end{equation}
The quantum Borel algebra $U^+_q (\mathfrak{so}_{2n})$ can be defined 
by the condition that the Hopf subalgebras $U_{n-1}$ and $U_n$ generated, respectively,
by $x_1,x_2,\ldots , x_{n-1}$ and  $x_1,x_2,\ldots , x_{n-2}, x_n$ are Hopf algebras
$U_q(\mathfrak{sl}_n)$ of type $A_{n-1},$ and by one additional relation
\begin{equation}
[x_{n-1},x_n]=0.
\label{Drelb}
\end{equation}

Recall that  $x_i,$ $n<i<2n$ denotes the generator $x_{2n-i},$
whereas if $1\leq i<2n,$ then $\phi (i)$
equals $2n-i,$ so that $x_i=x_{\phi (i)},$ see Definition \ref{fis}.

\begin{definition} \rm 
We define words $e(k,m),$ $1\leq k\leq m<2n$ in the following way:
\begin{equation}
e(k,m)=\left\{ \begin{matrix}
x_kx_{k+1}\cdots x_{m-1}x_m,\hfill & \hbox{ if } m<n \hbox{ or }  k>n;\cr
x_kx_{k+1}\cdots x_{n-2}x_nx_{n+1}\cdots x_m, \hfill & \hbox{ if } k<n-1<m;\hfill \cr
x_nx_{n+1}\cdots x_m,\hfill & \hbox{ if } k=n-1<m;\hfill \cr
x_nx_{n+2}x_{n+3}\cdots x_m,\hfill & \hbox{ if } k=n.\hfill
\end{matrix} \right.
\label{De}
\end{equation}
Respectively, $e(m,k)$ is the word opposite to $e(k,m).$ Further, we define 
a word $e^{\prime }(k,m)$ as a word that appears from $e(k,m)$ by replacing
the subword $x_nx_{n+1},$ if any, with $x_{n-1}x_n.$ 
Respectively, $e^{\prime }(m,k)$ is the word opposite to $e^{\prime }(k,m).$

We see that $e(k,m)$ coincides with $v(k,m)$ if $m<n$ or $k>n.$
If $k<n-1<m,$ then $e(k,m)$ appears from $v(k,m)$ by deleting the letter $x_{n-1}$
(but not of $x_{n+1}$!). Similarly, if $k=n,$ then $e(n,m)$ appears from $v(n,m)$ by deleting the letter $x_{n+1},$
whereas if $k=n-1,$ then we have $e(n-1,m)=v(n,m).$
We have to stress that according to this definition $e(n-1,n)=e(n,n)=e(n,n+1)=x_n.$
\label{Dsl}
\end{definition}

\begin{lemma}
If $1\leq k\leq m<2n,$ then 
\begin{equation}
p(e(k,m),e(k,m))=\sigma_k^m =\left\{ \begin{matrix}
q^2,\hfill &\hbox{if } m=\phi (k);\hfill \cr
q,\hfill &\hbox{otherwise}.\hfill 
                             \end{matrix}
                    \right.
\label{Dmu21}
\end{equation}
\label{Dsig}
\end{lemma}
\begin{proof}
If the word $e(k,m)$ does not contain a subword $x_nx_{n+1},$ then it belongs to either
$U_n$ or $U_{n-1}$ that are isomorphic to $U_{q}^+(\mathfrak{sl}_n).$ Hence we have 
$p(e(k,m),e(k,m))$ $=q.$

Let $k \leq n-1 <m.$ In this case $e(k,m+1)=e(k,m)x_{m+1}$ which allows one to use  
induction on $m-n+1.$ If $m=n,$ then $e(k,n)$ does not contain a sub-word $x_nx_{n+1}.$
Because $p(\hbox{-},\hbox{-})$ is a bimultiplicative map, we may decompose 
\begin{equation}
p(e(k,m+1),e(k,m+1))=\sigma _k^m\cdot q\cdot p(e(k,m),x_{m+1})\cdot p(x_{m+1}, e(k,m)).
\label{Dsi69}
\end{equation}
Using relations \ref{Db1rel} and \ref{Db1rell} we shall prove
\begin{equation}
p(e(k,m),x_{m+1})\cdot p(x_{m+1}, e(k,m))=\left\{ \begin{matrix}
1,\hfill &\hbox{if } k=\phi (m)-1;\hfill \cr
q^{-2},\hfill &\hbox{if } k=\phi (m);\hfill \cr
q^{-1},\hfill &\hbox{otherwise.}\hfill 
                             \end{matrix}
                    \right.
\label{Dsi45}
\end{equation}
The left hand side of the above equality is $\prod_{k\leq t\leq m,\, t\neq n-1} p_{t\, m+1}p_{m+1\, t}.$

If $m>n+1,$ then by \ref{Db1rel} and \ref{Db1rell} the factor $p_{t\, m+1}p_{m+1\, t}$ differs from 1 only
if $t\in \{ \phi (m)-2, \phi (m)-1, \phi (m), m \}$ and related values are respectively $q^{-1}, q^2, q^{-1},q^{-1}$
whereas the product of all those values is precisely $q^{-1}.$ Hence, if $k<\phi (m)-1,$ then the whole product is $q^{-1};$
if $k=\phi (m)-1,$ then this is $1;$ if $k=\phi (m),$ then this is $q^{-2};$ if $k>\phi (m),$ then this is again $q^{-1}.$
If $m=n+1,$ then nontrivial factors are related to $t\in \{ n-3, n-2, n, n+1\}$ with values $q^{-1}, q^2, q^{-1},q^{-1},$
respectively. Hence, we arrive to the same conclusion with 
$k<n-2=\phi (m)-1;$ $k=n-2=\phi (m)-1;$ and $k=n-1=\phi (m).$

Finally, if $m=n,$ then there is just one nontrivial factor which relates to $t=n-2$ with value $q^{-1},$
so that if $k\leq n-2=\phi (m)-2,$ then the total product is $q^{-1};$ if $k=n-1=\psi (m)-1,$ then this is $1.$
This completes the proof of (\ref{Dsi45}).

To complete the inductive step we use (\ref{Dsi45}) and inductive hypothesis:
if $k$ $=\phi (m)-1,$ then $\sigma _k^{m+1}$ $=q\cdot q\cdot 1$ $=q^2;$
 if $k$ $=\phi (m),$ then $\sigma _k^{m+1}$ $=q^2\cdot q\cdot q^{-2}$ $=q;$
otherwise $\sigma _k^{m+1}$ $=q\cdot q\cdot q^{-1}$ $=q.$
\end{proof}

\begin{lemma}
If the word $e(k,m)$ contains the subword $x_nx_{n+1};$ that is $k<n<m,$ then 
for each $i,$ $k\leq i<m$ we have 
\begin{equation}
p(e(k,i),e(i+1,m))\cdot p(e(i+1,m),e(k,i))=\sigma _k^m(\sigma _k^i \sigma _{i+1}^m)^{-1}=\mu _k^{m,i}.
\label{Dmu23}
\end{equation}
\label{Dmu}
\end{lemma}
\begin{proof}
If $k<n<m,$ then for $i\neq n-1$ there is a decomposition $e(k,m)=e(k,i)e(i+1,m)$ 
which implies (\ref{Dmu23}) because the form $p(\hbox{-},\hbox{-} )$ is bimultiplicative.
 For $i=n-1$ there is another equality 
$e^{\prime }(k,m)=e(k,i)e(i+1,m).$ Certainly  $p(e^{\prime }(k,m), e^{\prime }(k,m))$
$=p(e(k,m), e(k,m))$ $=\sigma _k^m.$ Hence (\ref{Dmu23}) is still valid.
\end{proof}

\smallskip

\

We define the bracketing of $e(k,m),$ $k\leq m<2n$ as follows.
\begin{equation}
e[k,m]=\left\{ \begin{matrix} [[[\ldots [x_k,x_{k+1}], \ldots ],x_{m-1}], x_m],\hfill 
&\hbox{if } m<\phi (k);\hfill \cr
 [x_k,[x_{k+1},[\ldots ,[x_{m-1},x_m]\ldots ]]],\hfill &\hbox{if } m>\phi (k);\hfill \cr 
[\! [e[k,m-1],x_m]\! ],\hfill &\hbox{if } m=\phi (k),\hfill 
\end{matrix}\right.
\label{Dww}
\end{equation}
where as above  $[\! [u,v]\! ]=uv-q^{-1}p(u,v)vu.$

Conditional identity (\ref{ind}) demonstrates that the value of  $e[k,m]$ in $U_q^+(\mathfrak{so}_{2n})$
is independent of the precise arrangement of brackets, provided that $m\leq n$ or $k\geq n.$
\begin{lemma} 
If  $k<n<m<\phi (k),$ then the value in $U_q^+(\mathfrak{so}_{2n})$ 
of the bracketed word $[y_kx_{n+1}x_{n+2}\cdots x_m],$ where $y_k=e[k,n],$
is independent of the precise arrangement of brackets. 
\label{Dins}
\end{lemma}
\begin{proof} 
To apply (\ref{ind}), it suffices to check 
$[y_k,x_t]=0,$  where $n+1<t\leq m$ or, equivalently, $\phi (m)\leq t<n-1.$ 
 We have
$$
[y_k,x_t]=\hbox{\Large[}[e[k,t-2],e[t-1,n]],x_t\hbox{\Large]}
=\hbox{\Large[}e[k,t-2], [e[t-1,n],x_t]\hbox{\Large]}.
$$
The polynomial $[e[t-1,n],x_t]$ is independent of $x_{n-1},$ so that it belongs to the Hopf subalgebra $U_n=U_{q}^+(\mathfrak{sl}_{n}).$
By \cite[Theorem $A_n$]{Kh4}, the element $[e[t-1,n],x_t]$ equals zero in $U_{q}^+(\mathfrak{sl}_{n})$
because the word $e(t-1,n)x_t$ is standard, and the standard bracketing is $[e[t-1,n],x_t].$ 
\end{proof}
\begin{lemma} 
If $k<n,$ $ \phi (k)<m,$ then the value in $U_q^+(\mathfrak{so}_{2n})$ 
of the bracketed word $[x_kx_{k+1}\cdots x_{n-2}x_ny_m],$ where $y_m=e[n+1,m],$
is independent of the precise arrangement of brackets. 
\label{Dins1}
\end{lemma}

\begin{proof}
To apply (\ref{ind}), we need the equalities $[x_t,y_m]=0,$ $k\leq t<n-1.$ 
The polynomial $[x_t,y_m]$ belongs to the 
subalgebra $U_{n-1}.$ 
Moreover,  $[x_t,y_m]$ is proportional to $[y_m, x_t]$ due to antisymmetry identity (\ref{cha}) because 
$p(x_t, y_m)p(y_m,x_t)$ $=p_{t\, t+1}p_{tt}p_{t\, t-1}\cdot p_{t+1\, t}p_{tt}p_{t-1\, t}$ $=1.$
The equality $[y_m, x_t]=0$ turns to the proved above equality $[e[k,n],x_t]=0$ if one renames  the variables
$x_{n+1}\leftarrow x_k,$ $x_{n+2}\leftarrow x_{k+1}, \ldots .$
\end{proof}

\section{PBW generators of $U_q^+(\mathfrak{so}_{2n})$}

\begin{proposition} If $q\neq -1,$ then
values of the elements  $e[k,m],$ $k\leq m<\phi (k)$ form a set of PBW generators with infinite heights
for the algebra $U_q^+(\mathfrak{so}_{2n})$ over {\bf k}$[G].$ 
\label{DstrB}
\end{proposition}
\begin{proof} All words $e(k,m),$ $k\leq m<\phi (k)$ are standard Lyndon-Shirshov words,
and by  \cite[Theorem $D_n,$ p. 225]{Kh4} under the standard bracketing, say $[e(k,m)],$ 
 they form a set of PBW generators with infinite heights.

By induction on $m-k$  we prove that the values in $U_q^+(\mathfrak{so}_{2n})$
of $[e(k,m)]$ equal the values of $e[k,m]$  with bracketing given in (\ref{Dww}).

If $m\leq n,$ then by  Lemma \ref{indle} we have nothing to prove.

If $k<n<m,$ then according to \cite[Lemma 7.25]{Kh4}, the brackets in $[e(k,m)]$ 
are set by the following recurrence formulae (we note that $[e(k, m)]=[e_{k\, \phi (m)}]$
 in the notations of \cite{Kh4}):
\begin{equation}
[e(k,m)]=\left\{ 
\begin{matrix}
[x_k[e(k+1, m)]], \hfill & \hbox{if  } m<\phi (k)-1; \hfill \cr
[[e(k, m-1)]x_m], \hfill & \hbox{if  } m=\phi (k)-1.\hfill 
\end{matrix} \right. 
\label{Dwsk}
\end{equation}
In the latter case the induction applies directly. In the former case using induction and Lemma \ref{Dins} we have 
$[e(k+1, m)]=e[k+1, m]=[e[k+1,n], e[n+1,m]].$ At the same time $[x_k,x_t]=0,$  $n< t\leq m$ because $x_t=x_{\phi (t)}$
and $\phi (t)\geq \phi (m)>k+1$ $\, \& \, (k,\phi (t))\neq (n-2,n).$
This implies $[x_k, e[n+1,m]]=0.$ Applying the conditional identity (\ref{jak3}), we get 
$$
[e(k, m)]=[x_k[e[k+1,n], e[n+1,m]]]={\big [}[x_ke[k+1,n]], e[n+1,m]{\big ]}=e[k,m].
$$
\end{proof}

\section{Shuffle representation for $U_q^+(\mathfrak{so}_{2n})$}

In this section, we are going to find the shuffle representation of 
elements $e[k,m],$ $1\leq k\leq m<2n.$ If $e(k,m)$ has not  
$x_nx_{n+1}$ as a subword, then $e[k,m]$ belongs to a Hopf subalgebra of type $A_n$: this is either 
$U_{n-1}$ $=U_q^+(\mathfrak{sl}_n)$ or $U_n$ $=U_q^+(\mathfrak{sl}_n).$
At the same time in the considered above case $C_n,$ the elements $x_1,x_2, \ldots , x_{n-1}$
generate precisely  a Hopf subalgebra $U_q(\mathfrak{sl}_n).$ Hence we may apply Proposition \ref{shu}:
\begin{equation}
e[k,m]=\alpha _k^m\cdot (e(m,k)),
\label{Dng0}
\end{equation}
where
\begin{equation}
\alpha _k^m=\left\{ 
\begin{matrix}
(q-1)^{m-k}\cdot \prod\limits _{k\leq i<j\leq m}p_{ij}, \hfill & \hbox{ if } m<n \hbox{ or } k>n; \hfill \cr
\ & \ \cr
(q-1)^{m-n-1}\cdot \prod\limits _{n\leq i<j\leq m,\,  i,j\neq n+1}p_{ij},\hfill  & \hbox{ if } k=n;\hfill \cr 
\ & \ \cr
(q-1)^{n-k-1}\cdot \prod\limits _{k\leq i<j\leq m,\,  i,j\neq n-1}p_{ij},\hfill & \hbox{ if } m=n.\hfill 
\end{matrix}
\right.
\label{Dng3}
\end{equation}

\begin{proposition}
Let $1\leq k<n<m<2n.$ In the shuffle representation, we have
\begin{equation}
e[k,m]=\alpha _k^m\cdot \{ (e(m,k))+p_{n-1,n}(e^{\prime }(m,k))\} , 
\label{Dshur}
\end{equation}
where
\begin{equation}
\alpha _k^m=\epsilon _k^m (q-1)^{m-k-1}\cdot \prod _{k\leq i<j\leq m,\,  i,j\neq n-1}p_{ij}
\label{Dshur1}
\end{equation}
with
\begin{equation}
\epsilon _k^m=\left\{ \begin{matrix} q^{-1},\hfill & \hbox{ if } m=\phi (k);\hfill \cr
1, \hfill & \hbox{otherwise.}\hfill 
\end{matrix}
\right.
\label{Dshur2}
\end{equation}
\label{Dshu}
\end{proposition}
\begin{proof} 

a). Consider first the case $m<\phi (k).$ We use induction on $m-n.$ Let $m-n=1.$
Condition $n+1=m<\phi (k)$ implies $k<n-1.$ Hence 
by Lemma \ref{Dins} we have $e[k,n+1]=[e[k,n],x_{n+1}],$
whereas (\ref{Dng0}) implies $e[k,n]=\alpha _k^n(e(n,k)).$ 
Using (\ref{spro}), we may write
$$
e[k,n+1]=\alpha _k^n\{ (e(n,k))(x_{n+1})-p(e(n,k),x_{n+1})\cdot (x_{n+1})(e(n,k))\}
$$
\begin{equation}
=\alpha _k^{n}\sum _{uv=e(n,k)}\{ p(x_{n+1},v)^{-1}-p(v,x_{n+1})\} (ux_{n+1}v),
\label{Dsum1}
\end{equation}
where $p(v,x_{n+1})=p(e(n,k),x_{n+1})p(u,x_{n+1})^{-1}$ because $e(n,k)=uv.$
We have 
$$
p(x_{n+1},v)^{-1}-p(v,x_{n+1})=p(v,x_{n+1})\cdot \{ p(x_{n+1},v)^{-1}p(v,x_{n+1})^{-1}-1 \}.
$$
At the same time equality $x_{n+1}=x_{n-1}$ and relations (\ref{Db1rel}), (\ref{Db1rell}) imply 
$$
p(x_{n+1},v)p(v,x_{n+1})=\left\{ \begin{matrix} q^{-1},\hfill & \hbox{ if } v=e(n,k) \hbox{ or } v=e(n-2,k);\hfill \cr
1, \hfill & \hbox{otherwise.}\hfill 
\end{matrix}
\right.
$$ 
Hence in the decomposition (\ref{Dsum1}) two terms remain
$$
\alpha _k^np(e(n,k),x_{n+1})(q-1)(x_{n+1}x_nx_{n-2}\cdots x_k)=\alpha _k^{n+1} (e(n+1,k))
$$
and
$$
\alpha _k^np(e(n-2,k),x_{n+1})(q-1)(x_nx_{n-1}x_{n-2}\cdots x_k)=\alpha _k^{n+1} p_{n-1,n}(e^{\prime }(n+1,k)),
$$
for $p_{n,n+1}^{-1}=p_{n,n-1}^{-1}=p_{n-1,n}$ due to (\ref{Db1rell}). This completes the first step of induction.

Suppose that equalities (\ref{Dshur}) and (\ref{Dshur1}) are valid and still $m+1<\phi (k).$ Then Lemma  \ref{Dins} implies 
$e[k,m+1]=[e[k,m],x_{m+1}].$ By (\ref{spro}) we have 
$$
[(e(m,k)),(x_{m+1})]=\sum _{uv=e(m,k)}p(v,x_{m+1})\cdot \{ p(x_{m+1},v)^{-1}p(v,x_{m+1})^{-1}-1 \}(ux_{m+1}v).
$$
Relations (\ref{Db1rel}), (\ref{Db1rell}) imply that 
$$
p(x_{m+1},v)p(v,x_{m+1})=\left\{ \begin{matrix} q, \hfill & \hbox{ if } v=e(\phi (m)-1,k);\hfill \cr
q^{-1}, \hfill & \hbox{ if } v=e(m,k) \hbox{ or } v=e(\phi (m)-2,k) ;\hfill \cr
1,  \hfill & \hbox{ otherwise.}\hfill
\end{matrix}\right.
$$
Thus in the decomposition just three terms remain. Two of them, corresponding to 
$v=e(\phi (m)-1,k)$ and $v=e(\phi (m)-2,k),$ are canceled:
$$
p(x_{\phi (m)-1},x_{m+1})(q^{-1}-1)+(q-1)=q(q^{-1}-1)+(q-1)=0.
$$
Thus
$$
[(e(m,k)),(x_{m+1})]=\{ (q-1)\cdot \prod _{k\leq i\leq m,\ i\neq n-1}p_{i\, m+1} \} \, (e(m+1,k)).
$$
In perfect analogy, we have
$$
[(e^{\prime }(m,k)),(x_{m+1})]=\{ (q-1)\cdot \prod _{k\leq i\leq m,\ i\neq n-1}p_{i\, m+1} \} \, (e^{\prime }(m+1,k)).
$$
The inductive supposition yields $e[m,k]=\alpha _k^m\cdot \{(e(m,k))+p_{n-1,n}(e^{\prime }(k,m))\} .$
Hence to complete the induction, it suffices to note that 
$$\alpha _k^{m+1}=\alpha _k^m\cdot (q-1)\cdot \prod _{k\leq i\leq m,\ i\neq n-1}p_{i\, m+1}.$$

b). Similarly consider the case $m>\phi (k)$ using downward induction on $n-k.$
Let $k=n-1.$ Condition $m>\phi (k)$ implies $m\geq n+2.$ Hence 
by Lemma \ref{Dins1} we have $e[n-1,m]=[x_n,e[n+1,m]],$
whereas (\ref{Dng0}) and (\ref{Dng3}) imply $e[n+1,m]$ $=\alpha _{n+1}^m(e(m,n+1)).$ 
Using (\ref{spro}), we may write
$$
e[n-1,m]=\alpha _{n+1}^m \{ (x_n)(e(m,n+1))-p(x_n,e(m,n+1))\cdot (e(m,n+1))(x_n)\}
$$
\begin{equation}
=\alpha _{n+1}^{m}\sum _{uv=e(m,n+1)}\{ p(u,x_n)^{-1}-p(x_n,u)\} (ux_nv),
\label{Dsum2}
\end{equation}
where $p(x_n,u)=p(x_n, e(m,n+1))p(x_{n},v)^{-1}$ because $e(m,n+1)=uv.$
We have 
$$
p(u,x_n)^{-1}-p(x_n,u)=p(x_n,u)\cdot \{ p(u,x_n)^{-1}p(x_n,u)^{-1}-1 \}.
$$
Equality $x_{n+1}=x_{n-1}$ and relations (\ref{Db1rel}), (\ref{Db1rell})
imply that
$p(u,x_n)p(x_n,u)=1$ unless $u=e(m,n+1)$ or $u= e(m,n+2).$
In these two exceptional cases, the product equals $p_{n+2\, n}p_{n\, n+2}=q^{-1}$ because $p_{n\, n+1}p_{n+1\, n}=1.$ 
Hence in the decomposition (\ref{Dsum2})  two terms remain
$$
\alpha _{n+1}^{m}p(x_n,e(m,n+1))(q-1)(x_m\cdots x_{n+1}x_n)=\alpha _{n-1}^{m}(e(m,n-1))
$$
and
$$
\alpha _{n+1}^{m}p(x_n,e(m,n+2))(q-1)(x_m\cdots x_{n+2}x_nx_{n+1})=\alpha _{n-1}^{m}\cdot p_{n\, n+1}^{-1}(e^{\prime }(m,n-1)).
$$
This completes the first step of induction because $p_{n\, n+1}^{-1}=p_{n-1\, n}.$

Suppose that equalities  (\ref{Dshur}) and (\ref{Dshur})  are valid and still $m>\phi (k-1)=\phi (k)+1.$ Lemma  \ref{Dins1} implies 
$e[k-1,m]=[x_{k-1},e[k,m]].$ We have  
$$
[(x_{k-1}),(e(m,k))]=\sum _{uv=e(m,k)}p(x_{k-1},u)\cdot \{ p(u,x_{k-1})^{-1}p(x_{k-1},u)^{-1}-1 \}(ux_{k-1}v).
$$
Relations (\ref{Db1rel}), (\ref{Db1rell}) imply that 
$$
p(u,x_{k-1})p(x_{k-1},u)=\left\{ \begin{matrix} q \hfill & \hbox{ if } u=e(m, \phi (k)+1);\hfill \cr
q^{-1} \hfill & \hbox{ if } u=e(m,k) \hbox{ or } u=e(\phi (k)+2,k) ;\hfill \cr
1  \hfill & \hbox{ otherwise.}\hfill
\end{matrix}\right.
$$
Hence in the decomposition (\ref{Dsum1}) just three terms remain. Two of them, corresponding to 
$u=e(m, \phi (k)+1)$ and $u=e(m,\phi (k)),$ are canceled:
$$
p(x_{k-1},x_{\phi (k)+1})(q^{-1}-1)+(q-1)=q(q^{-1}-1)+(q-1)=0.
$$
Thus
$$
[(x_{k-1}),(e(m,k)]=\{ (q-1)\cdot \prod _{k\leq j\leq m,\ j\neq n-1}p_{k-1\,j} \} \, (v(m,k-1)).
$$
In perfect analogy, we have
$$
[(x_{k-1}),(e^{\prime }(m,k)]=\{ (q-1)\cdot \prod _{k\leq j\leq m,\ j\neq n-1}p_{k-1\,j} \} \, (v^{\prime }(m,k-1)).
$$
The inductive supposition states $e[m,k]=\alpha _k^m\cdot \{(e(m,k))+p_{n-1,n}(e^{\prime }(k,m))\} .$
Hence it remains to note that 
$$\alpha _{k-1}^{m}=\alpha _k^m\cdot (q-1)\cdot \prod _{k\leq j\leq m,\ j\neq n-1}p_{k-1\, j}.$$

c). Let $m=\phi (k)\neq n.$ In this case, $x_m=x_k,$ $\epsilon_k^m=q^{-1}.$ If $k=n-1,$ $m=n+1,$ then $e(n-1,n)=x_n$ and 
by Definition \ref{Dww} we have 
$$
e[n-1,n+1]=x_nx_{n+1}-q^{-1}p_{n\, n+1}x_{n+1}x_n=(1-q^{-1})x_nx_{n+1}
$$
since due to (\ref{Drelb}) we have  $x_{n+1}x_n=p_{n-1\, n}x_nx_{n+1}$ with $x_{n+1}=x_{n-1}$
and $p_{n\, n+1}p_{n-1\, n}$ $=1.$ In the shuffle form, we get
$$
(x_n)(x_{n+1})=(x_nx_{n+1})+p_{n+1\, n}^{-1}(x_{n+1}x_n)=p_{n\, n+1}\cdot \{ (x_{n+1}x_n)+p_{n-1\, n}(x_nx_{n+1})\}.
$$
It remains to note that $e(n-1, n+1)$ $=x_nx_{n+1},$ $e(n+1,n-1)$ $=x_{n+1}x_n,$ 
$e^{\prime }(n-1, n+1)$ $=x_{n-1}x_n,$ $e^{\prime }(n+1,n-1)$ $=x_nx_{n-1}$ $=x_nx_{n+1}.$

 Let $k<n-1.$ By definition (\ref{Dww}) we have
\begin{equation}
e[k,m]=e[k,m-1]\cdot x_m-q^{-1}p(e(k,m-1),x_m)x_k\cdot e[k,m-1].
\label{Dkk}
\end{equation}
Already done case a)  allows us to find the shuffle representation
$$
e[k,m-1]=\alpha _k^{m-1} \cdot \{ (e(m-1,k))+p_{n-1\, n}(e^{\prime }(m-1,k))\} .
$$
We have 
$$
[\![(e(m-1,k)),(x_m)]\!]=\sum _{uv=e(m-1,k)}p(v,x_{m})\cdot \{ p(x_{m},v)^{-1}p(v,x_{m})^{-1}-q^{-1} \}(ux_{m}v).
$$
Relations (\ref{Db1rel}), (\ref{Db1rell}) imply that 
$$
p(x_{m},v)p(v,x_{m})=\left\{ \begin{matrix} 1, \hfill & \hbox{ if } v=\emptyset  \hbox{ or }v=e(m-1,k);\hfill \cr
q^2, \hfill & \hbox{ if } v=x_k;\hfill \cr
q,  \hfill & \hbox{ otherwise.}\hfill
\end{matrix}\right.
$$
Therefore in the decomposition just three terms remain. Two of them, corresponding to 
$v=\emptyset $ and $v=x_k,$ are canceled:
$$
p(x_k,x_{m})(q^{-2}-q^{-1})+(1-q^{-1})=q(q^{-2}-q^{-1})+(1-q^{-1})=0.
$$
Thus
$$
[\![(e(m-1,k)),(x_{m})]\!] =\{ (1-q^{-1})\cdot \prod _{k\leq i<m,\ i\neq n-1}p_{i\, m} \} \, (e(m,k)).
$$
In perfect analogy, we have
$$
[\![(e^{\prime }(m-1,k)),(x_{m})]\!] =\{ (1-q^{-1})\cdot \prod _{k\leq i<m,\ i\neq n-1}p_{i\, m} \} \, (e^{\prime }(m,k)).
$$
It suffices to note that $1-q^{-1}=\epsilon _k^m (q-1),$ and by definition
$$\alpha _k^{m}=\alpha _k^{m-1}\cdot \epsilon _k^m (q-1)\cdot \prod _{k\leq i<m,\ i\neq n-1}p_{i\, m}.$$
 The proposition is completely proved.
\end{proof} 

\section{Coproduct formula for $U_q(\mathfrak{so}_{2n})$}

\begin{theorem} In $U_q^+(\mathfrak{so}_{2n})$
the coproduct on the elements $e[k,m],$ $k\leq m<2n$ 
has the following explicit form
\begin{equation}
\Delta (e[k,m])=e[k,m]\otimes 1+g_{km}\otimes e[k,m]
\label{Dco}
\end{equation}
$$
+\sum _{i=k}^{m-1}\tau _i(1-q^{-1})g_{ki}\, e[i+1,m]\otimes e[k,i],
$$
where $\tau _i=1,$ with two  exceptions, being $\tau _n=0$ if $k=n,$ and $\tau _{n-1}=0$ if $m=n;$
 and $\tau _{n-1}=p_{n\, n-1}$ otherwise. Here $g_{ki}={\rm gr}(e(k,i))$ is a group-like element 
that appears from the word $e(k,i)$ under the substitutions $x_{\lambda }\leftarrow g_{\lambda }.$
\label{Dcos}
\end{theorem}
\begin{proof}
If the word $e(k,m)$ does not contain the subword $x_{n}x_{n+1},$ then $e[k,m]$ belongs to either $U_{n-1}$ or $U_n.$ Both of these
Hopf algebras  are isomorphic to $U_q^+(\mathfrak{sl}_{n}).$ Hence if $m\leq n$ or $k\geq n,$  then we have nothing to prove.

Suppose that $k<n<m.$ In this case by Proposition \ref{Dshu} we have the shuffle representation
\begin{equation}
e[k,m]=\alpha _k^m\cdot \{ (e(m,k))+p_{n-1\, p}(e^{\prime }(m,k))\} ,
\label{Dco2}
\end{equation}
where $(e(m,k))$ is a comonomial shuffle 
$$
(e(m,k))=\left\{ \begin{matrix}(x_mx_{m-1}\cdots x_{n+2}x_{n+1}x_nx_{n-2}\cdots x_k), \hfill & \hbox{ if } k<n-1; \hfill \cr
(x_mx_{m-1}\cdots x_{n+2}x_{n+1}x_n), \hfill & \hbox{ if } k=n-1, \hfill 
\end{matrix}
\right.
$$
whereas $(e^{\prime }(m,k))$ is a related  one:
$$
(e^{\prime }(m,k))=\left\{ \begin{matrix}(x_mx_{m-1}\cdots x_{n+2}x_nx_{n-1}x_{n-2}\cdots x_k), \hfill & \hbox{ if } k<n-1; \hfill \cr
(x_mx_{m-1}\cdots x_{n+2}x_nx_{n-1}), \hfill & \hbox{ if } k=n-1. \hfill 
\end{matrix}
\right.
$$
Using (\ref{bcopro}) it is easy to find the braided coproduct of the comonomial shuffles:  
$$
\Delta ^b_0((e(m,k)))=\sum _{i=k}^{n-2}(e(m,i+1))\underline{\otimes }(e(i,k))
+\sum _{i=n}^{m-1}(e(m,i+1))\underline{\otimes }(e(i,k)),
$$
$$
\Delta ^b_0((e^{\prime }(m,k)))=\sum _{i=k}^{n-1}(e^{\prime }(m,i+1))\underline{\otimes }(e(i,k))
+\sum _{i=n+1}^{m-1}(e(m,i+1))\underline{\otimes }(e^{\prime }(i,k)),
$$
where for short we define $\Delta _0^b(U)=\Delta^b(U)-U\underline{\otimes }1-1\underline{\otimes }U.$
Taking into account (\ref{Dco2}), we have
$$
(\alpha _k^m)^{-1}\Delta ^b_0(e[k,m])=\left( \sum _{i=k}^{n-2} (\alpha _{i+1}^m)^{-1}e[i+1,m]\underline{\otimes }(e(i,k))\right) 
+p_{n-1\, n}(e(m,n))\underline{\otimes }(e(n-1,k))
$$
$$
+(e(m,n+1)) \underline{\otimes } (e(n,k))+\sum _{i=n+1}^{m-1}(e(m,i+1))\underline{\otimes } (\alpha _k^i)^{-1}e[k,i].
$$
Relation (\ref{Dng0}) applied to $e[k,i],$ $i\leq n$ and $e[i+1,m],$ $i\geq n$ allows one to rewrite 
the right hand side of the above equality in terms of $e[i,j]:$
$$
=\left( \sum _{i=k}^{n-2}(\alpha _k^i)^{-1}(\alpha _{i+1}^m)^{-1}e[i+1,m]\underline{\otimes }e[k,i]\right) 
+p_{n-1\, n}(\alpha _n^m)^{-1}(\alpha _k^{n-1})^{-1}e[n,m]\underline{\otimes }e[k,n-1]
$$
$$
+ (\alpha _{n+1}^m)^{-1} (\alpha _k^n)^{-1} e[n+1,m] \underline{\otimes } e[k,n]+
\sum _{i=n+1}^{m-1}(\alpha _{i+1}^m)^{-1}(\alpha _k^i)^{-1}e[i+1,m]\underline{\otimes }e[k,i].
$$
Thus, we have
\begin{equation}
\Delta ^b_0(e[k,m]) =\sum _{i=k}^{m-1}\gamma _i\, e[i+1,m]\underline{\otimes }e[k,i],
\label{Dco5}
\end{equation}
where 
\begin{equation}
\gamma _i=p_{n-1\,n}^{\delta _{n-1}^i} \cdot \alpha _k^m  (\alpha _k^i \alpha _{i+1}^m)^{-1},
\label{Dco6}
\end{equation}
whereas $\delta _{n-1}^i$ is the Kronecker delta.

Our next step is to see that for all $i,$ $k\leq i<m$ we have
\begin{equation}
\gamma _i=p_{n\, n-1}^{\delta _{n-1}^i} \cdot (q-1)\epsilon _k^m  (\epsilon _k^i \epsilon _{i+1}^m)^{-1}p(e(k,i),e(i+1,m)).
\label{gam}
\end{equation}
All factors except the $\epsilon $'s in (\ref{Dco6}) have the form $(q-1)^s\prod _{A}p_{ab},$ where $A$ is a suitable set of
pairs $(a,b)$ and $s$ is an integer exponent. Due to bimultiplicativity of the form $p(\hbox{-},\hbox{-}),$
the same is true for the right hand side of (\ref{gam}). Hence it suffices to demonstrate that the sum of exponents 
of the factors in (\ref{Dco6}) equals $1,$ and the resulting product domains  in (\ref{Dco6}) and (\ref{gam})
 are the same, or at least they define the same product.

If $i<n-1,$ then using (\ref{Dco6}),  (\ref{Dng0}), and Proposition \ref{Dshu}, we have
the required equality for the exponents,
$$
(m-k-1)-(i-k)-(m-i-2)=1,
$$
and for the product domains:
$$
\{k \leq a<b\leq m,\,  a,b\neq n-1\} \setminus (\{ k\leq a<b\leq i\} \cup \{ i+1\leq a<b\leq m,\,  a,b \neq n-1\} )
$$
$$
=\{k\leq a\leq i <b\leq m,\,  a,b \neq n-1\} .
$$

If $i\geq n,$ then similarly we have the required equality for the exponents,
$$
(m-k-1)-(i-k-1)-(m-i-1)=1,
$$
and for the product domains:
$$
\{k \leq a<b\leq m,\,  a,b\neq n-1\} \setminus (\{ k\leq a<b\leq i, \,  a,b \neq n-1\} \cup \{ i+1\leq a<b\leq m, \} )
$$
$$
=\{k\leq a\leq i <b\leq m,\,  a,b \neq n-1\} .
$$

In the remaining case, $i=n-1,$ we have $e(k,i)$ $=x_k\cdots x_{k-2}x_{k-1},$
$e(i+1,m)$ $=x_nx_{n+2}\cdots x_m.$ Due to (\ref{Dng3}),  the exponent is
$$
(m-k-1)-(n-1-k)-(m-n-1)=1,
$$
whereas the product domain of $\alpha _k^m  (\alpha _k^{n-1}\alpha _n^m)^{-1}$ reduces to 
$$
\{k \leq a<b\leq m,\,  a,b\neq n-1\} \setminus (\{ k\leq a<b\leq n-1\} \cup \{ n\leq a<b\leq m,\, a,b\neq n+1 \} )
$$
\begin{equation}
=\{k\leq a<n-1< b\leq m\}  \cup \{ n+1=a<b\leq m\} \cup \{ (n,n+1)\}.
\label{Dco9}
\end{equation}
However, in this case the product domain for $\alpha _k^{n-1}$
is not a subset of the product domain for $\alpha _k^{m}.$ Therefore additionally to the product defined by (\ref{Dco9}),
there appears a factor $\prod\limits_{k\leq a<b=n-1}p_{ab}^{-1}$ and a factor $p_{n-1\, n}$
 that comes from (\ref{Dco6}) due to $\delta _{n-1}^i=1.$
The latter factor cancels  with the factor defined by the subset $ \{ (n,n+1)\}$ since $p_{n-1\, n}p_{n\, n+1}=1,$
whereas the product domain of the former factor must be added to  the product domain of $p(e(k,i),e(i+1,m))$:
\begin{equation}
\{ k\leq a\leq n-1<b\leq m,\, b\neq n+1\} \cup  \{k\leq a<b=n-1\}.
\label{Dco11}
\end{equation}
It remains to compare the products defined by (\ref{Dco9}) without the last pair and that defined by  (\ref{Dco11}).

The set $\{ k\leq a<n-1<b\leq m,\, b\neq n+1 \}$ is a subset of the first sets in (\ref{Dco9}) and  (\ref{Dco11}).
After cancelling the pairs from that set, (\ref{Dco9}) and  (\ref{Dco11})  transform to, respectively,
\begin{equation}
\{k\leq a<n-1< b=n+1\}  \cup \{ n+1=a<b\leq m\}
\label{Dco91}
\end{equation}
and
\begin{equation}
\{ a=n-1<b\leq m,\, b\neq n+1\} \cup  \{k\leq a<b=n-1\}.
\label{Dco12}
\end{equation}
The first set of (\ref{Dco91}) and the second set of (\ref{Dco12}) define the same product 
because $x_{n+1}=x_{n-1}$ and $p_{a\, n+1}=p_{a\, n-1}.$ By the same reason
$p_{n-1\, b}=p_{n+1\, b},$ hence  the first set  of (\ref{Dco12}) defines the same product
as the second set of (\ref{Dco91}) up to one additional factor, $p_{n-1\, n}$ that corresponds to the pair  $(n-1,n).$
 This factor  is canceled by the first factor $p_{n,n-1}$ that appears in (\ref{gam}) due to
$\delta _{n-1}^i=1.$ The equality (\ref{gam}) is completely proved.

\smallskip

Now we are ready to consider the (unbraided) coproduct. Formula  (\ref{copro})
demonstrates that the tensors $u^{(1)}\otimes u^{(1)}$ of the coproduct and tensors
$u^{(1)}_b\underline{\otimes } u^{(1)}_b$ of the braided coproduct are related by 
$u^{(1)}_b=u^{(1)}{\rm gr}(u^{(2)})^{-1},\ $ $u^{(2)}_b=u^{(2)}.$ The equality  (\ref{Dco5})
provides the values of $u^{(1)}_b$ and $u^{(2)}_b.$ Hence we may find  
$u^{(1)}=\gamma _ie[i+1,m]g_{ki}$ and $u^{(2)}=e[k,i],$ where $g_{ki}={\rm gr}(e[k,i]).$ 
The commutation rules imply $$e[i+1,m]g_{ki}=p(e(i+1,m), e(k,i))g_{ki}e[i+1,m].$$
Therefore the coproduct has the form (\ref{Dcos}), where  $$\tau _i(1-q^{-1})=\gamma _i \, p(e(i+1,m), e(k,i)).$$
 Applying (\ref{gam}) and Lemma \ref{Dmu} we get
$$
\tau _i(1-q^{-1})=p_{n\, n-1}^{\delta _{n-1}^i} \cdot (q-1)\epsilon _k^m  (\epsilon _k^i \epsilon _{i+1}^m)^{-1}
\cdot \sigma _k^m  (\sigma_k^i \sigma_{i+1}^m)^{-1}.
$$
Lemma \ref{Dsig} and Eq.  ( \ref{Dshur2}) imply that $\epsilon _k^m\sigma _k^m$ equals $q$ for all $k,m$
without exceptions. Hence $$\epsilon _k^m  (\epsilon _k^i \epsilon _{i+1}^m)^{-1}
\cdot \sigma _k^m  (\sigma_k^i \sigma_{i+1}^m)^{-1}=q^{-1},$$
and we have
\begin{equation}
\tau_i=p_{n\, n-1}^{\delta _{n-1}^i}
=\left\{ \begin{matrix}
p_{n\, n-1},\hfill &\hbox{if } i=n-1;\hfill \cr
1,\hfill &\hbox{otherwise}.\hfill 
                             \end{matrix}
                    \right.
\label{Dtau}
\end{equation}
The theorem is completely proved.
\end{proof}

{\bf Remark 2.} If $q^t=1,$ $t>2,$ then 
(\ref{Dco}) remains valid due to precisely the same arguments
 that were given in Remark 1, see page \pageref{t}.

{\bf Remark 3.} 
In fact,  the exceptions
$\tau _n=0$ if $k=n,$ and $\tau _{n-1}=0$ if $m=n$ can be omitted  in the statement of the above theorem.  Indeed,
the related tensors are,  respectively, $e[n+1,m]\otimes e[n,n]$ and $e[n,n]\otimes e[k,n-1],$
whereas by definition $e[n,n]=[\![x_n,x_n]\!] =x_n\cdot x_n-q^{-1}p(x_n,x_n)x_n\cdot x_n=0.$
So that, we may assume $\tau_n =1,$ $\tau _{n-1}=p_{n\, n-1}$ as well.

\end{document}